\documentclass[11pt]{amsart}
\usepackage{amsmath, amssymb, amsfonts, amsthm, verbatim, dsfont, graphicx}
\usepackage{hyperref}
\usepackage{color}

\newtheorem{theorem}{Theorem}[section]
\newtheorem{proposition}[theorem]{Proposition}

\newtheorem{lemma}[theorem]{Lemma}
\newtheorem{definition}[theorem]{Definition}
\theoremstyle{definition} \newtheorem{remark}[theorem]{Remark}

\newcommand{\bR}{{\mathbb R}}
\newcommand{\bN}{{\mathbb N}}
\newcommand{\bC}{{\mathbb C}}
\newcommand{\bT}{{\mathbb T}}
\newcommand{\bZ}{{\mathbb Z}}

\newcommand{\bP}{{\mathbb P}}

\newcommand{\cF}{{\mathcal{F}}}

\newcommand{\cH}{{\mathcal{H}}}

\begin{document}

\title[]{Random data Cauchy theory for nonlinear wave equations of power-type on $\bR^3$}
\begin{abstract}
 We consider the defocusing nonlinear wave equation of power-type on $\bR^3$. We establish an almost sure global existence result with respect to a suitable randomization of the initial data. In particular, this provides examples of initial data of super-critical regularity which lead to global solutions. The proof is based upon Bourgain's high-low frequency decomposition and improved averaging effects for the free evolution of the randomized initial~data.
\end{abstract}

\author[]{Jonas L\"uhrmann}
\address{Departement Mathematik \\ ETH Z\"urich \\ 8092 Z\"urich \\ Switzerland}
\email{jonas.luehrmann@math.ethz.ch}

\author[]{Dana Mendelson}
\address{Department of Mathematics \\ Massachusetts Institute of Technology \\ 77 Massachusetts Ave \\ Cambridge, MA 02139 \\ USA}
\email{dana@math.mit.edu}

\thanks{The first author was supported in part by the Swiss National Science Foundation under grant SNF 200021-140467. The second author was supported in part by the U.S. National Science Foundation grant DMS-1068815 and by the NSERC Postgraduate Scholarships Program}

\maketitle

\section{Introduction}
\setcounter{equation}{0}

We consider the Cauchy problem for the defocusing nonlinear wave equation
\begin{equation} \label{equ:ivp}
 \left\{ \begin{split}
  -u_{tt} + \Delta u &= |u|^{\rho-1} u \text{ on } \bR \times \bR^3, \\
  (u, u_t)|_{t=0} &= (f_1, f_2) \in H^s_x(\bR^3) \times H^{s-1}_x(\bR^3),
 \end{split} \right.
\end{equation}
where $3 \leq \rho \leq 5$ and $H^s_x(\bR^3)$ is the usual inhomogeneous Sobolev space. The main result of this paper establishes the almost sure existence of global solutions to~\eqref{equ:ivp} with respect to a suitable randomization of initial data in $H^s_x(\bR^3) \times H^{s-1}_x(\bR^3)$ for $3 \leq \rho < 5$ and $\frac{\rho^3+5\rho^2-11\rho-3}{9\rho^2-6\rho-3} < s < 1$. In particular, for $\frac{1}{4}(7+\sqrt{73}) < \rho < 5$, this yields examples of initial data of super-critical regularity for which solutions to \eqref{equ:ivp} exist globally in time. See Remark~\ref{remark:uniqueness} for a discussion of uniqueness of these solutions.

\medskip

A systematic investigation of the local well-posedness of \eqref{equ:ivp} for initial data in homogeneous Sobolev spaces is undertaken by Lindblad and Sogge in \cite{Lindblad_Sogge}, where local strong solutions to \eqref{equ:ivp} are constructed for $s \geq \frac{3}{2} - \frac{2}{\rho-1}$ using Strichartz estimates for the wave equation. When $s < \frac{3}{2} - \frac{2}{\rho-1}$, this is the super-critical regime and the well-posedness arguments based on Strichartz estimates break down. Recently, several methods have emerged proving ill-posedness for \eqref{equ:ivp} below the critical scaling regularity (see Lebeau \cite{L}, Christ-Colliander-Tao \cite{CCT} and Ibrahim-Majdoub-Masmoudi \cite{IMM}).

\medskip

Nevertheless, using probabilistic tools, it has been possible to construct large sets of initial data of super-critical regularity that lead to unique local and even global solutions to several nonlinear dispersive equations. This approach was initiated by Bourgain \cite{B94, B96} for the periodic nonlinear Schr\"odinger equation in one and two space dimensions. Building upon work by Lebowitz, Rose and Speer \cite{LRS}, Bourgain proved the invariance of the Gibbs measure under the flow of the equation and used this invariance to prove almost sure global well-posedness in the support of the measure. In \cite{BT1, BT2}, Burq and Tzvetkov consider the cubic nonlinear wave equation on a three-dimensional compact manifold. They construct large sets of initial data of super-critical regularity that lead to local solutions via a randomization procedure which relies on expansion of the initial data with respect to an orthonormal basis of eigenfunctions of the Laplacian. Together with invariant measure considerations they also prove almost sure global existence for the cubic nonlinear wave equation on the three-dimensional unit ball. Many further results in this direction have been obtained in recent years (see \cite{Tz08}, \cite{Tz10}, \cite{BT3}, \cite{TTz}, \cite{O2}, \cite{NORS}, \cite{NRSS}, \cite{Ric}, \cite{BB} and references therein).

\medskip

Although it is desirable to establish the existence of an invariant Gibbs measure, in many situations there are serious technical difficulties associated with defining such a measure. In the absence of an invariant measure, other approaches have been developed to show almost sure global existence for periodic super-critical equations via a suitable randomization of the initial data. Energy estimates are used by Nahmod, Pavlovi\'c and Staffilani \cite{NPS} in the context of the periodic Navier-Stokes equation in two and three dimensions and by Burq and Tzvetkov for the three-dimensional periodic defocusing cubic nonlinear wave equation. Colliander and Oh \cite{CO} adapt Bourgain's high-low frequency decomposition \cite{B98} to prove almost sure global existence of solutions to the one-dimensional periodic defocusing cubic nonlinear Schr\"odinger equation below $L^2(\mathbb{T})$. 

\medskip

In this work we study almost sure global existence for nonlinear wave equations on Euclidean space without any radial symmetry assumption on the initial data. Previous results on Euclidean space have involved first considering a related equation in a setting where an orthonormal basis of eigenfunctions of the Laplacian exists. This orthonormal basis is used to randomize the initial data for the related equation and an appropriate transform is then used to map solutions of the related equation to solutions of the original equation. Burq, Thomann and Tzvetkov \cite[Theorem 1.2]{BTT} prove almost sure global existence and scattering for the defocusing nonlinear Schr\"odinger equation on $\bR$ by first treating a one-dimensional nonlinear Schr\"odinger equation with a harmonic potential and then invoking the lens transform \cite[(10.2)]{BTT}. Subsequently, similar approaches were used by Deng \cite[Theorem 1.2]{D1}, Poiret \cite{Poiret1}, \cite{Poiret2} and Poiret, Robert and Thomann \cite[Theorem 1.3]{PRT} to study the defocusing nonlinear Schr\"odinger equation on $\bR^d$ for $d \geq 2$. In the context of the wave equation, Suzzoni \cite{Suzzoni1}, \cite[Theorem 1.2]{Suzzoni2} first considers a nonlinear wave equation on the three-dimensional sphere and then uses the Penrose transform \cite[(3)]{Suzzoni1} to obtain almost sure global existence and scattering for \eqref{equ:ivp} for $3 \leq \rho < 4$. 

\medskip

Instead of using such transforms, we will randomize functions directly on Euclidean space via a unit-scale decomposition in frequency space. More precisely, let $\varphi \in C_c^{\infty}(\bR^3)$ be a real-valued, smooth, non-increasing function such that $0 \leq \varphi \leq 1$ and
\begin{equation*}
 \varphi(\xi) = \left\{ 
  \begin{split}
   &1 \qquad \text{for } |\xi| \leq 1, \\
   & 0 \qquad \text{for } |\xi| \geq 2.
  \end{split} \right.
\end{equation*}
For every $k \in \bZ^3$ set $\varphi_k(\xi) = \varphi(\xi - k) $ and define
\begin{equation*}
 \psi_k(\xi) = \frac{\varphi_k(\xi)}{\sum_{l \in \bZ^3} \varphi_l(\xi)}.
\end{equation*}
Then $\psi_k$ is smooth with support contained in $\{ \xi \in \bR^3 : |\xi - k| \leq 2 \}$. Note that $\sum_{k \in \bZ^3} \psi_k(\xi) = 1$ for all $\xi \in \bR^3$. For $f \in L^2_x(\bR^3)$ define the function $P_k f: \bR^3 \rightarrow \bC$~by
 \begin{equation*}
  (P_k f)(x) = \cF^{-1} \left( \psi_k(\xi) \hat{f}(\xi) \right)(x) \text{ for } x \in \bR^3.
 \end{equation*} 
If $f \in H^s_x(\bR^3)$ for some $s \in \bR$, then $P_k f \in H^s_x(\bR^3)$ and $f = \sum_{k \in \bZ^3} P_k f \text{ in } H^s_x(\bR^3)$ with
 \begin{equation} \label{eq:norm_equiv}
  \|f\|_{H^s_x(\bR^3)} \sim \bigl( \sum_{k \in \bZ^3} \|P_k f\|^2_{H^s_x(\bR^3)} \bigr)^{1/2}.
 \end{equation}
In the proof of the almost sure existence of global solutions to \eqref{equ:ivp}, we will crucially exploit that these projections satisfy a unit-scale Bernstein inequality, namely that for all $2 \leq p_1 \leq p_2 \leq \infty$ there exists $C \equiv C(p_1, p_2) > 0$ such that for all $f \in L^2_x(\bR^3)$ and for all $k \in \bZ^3$ 
\begin{align} \label{equ:unit_scale_intro}
  \|P_k f\|_{L^{p_2}_x(\bR^3)} \leq C \|P_k f\|_{L^{p_1}_x(\bR^3)}.
\end{align}
For a formal statement of this inequality, see Lemma \ref{lem:unit_bernstein}. Let now $\{ (h_k, l_k) \}_{k \in \bZ^3}$  be a sequence of independent, zero-mean, real-valued random variables on a probability space $(\Omega, {\mathcal A}, \bP)$ with distributions $\mu_k$ and $\nu_k$. Assume that there exists $c > 0$ such that
 \begin{equation} \label{eq:rvassumption}
  \left| \int_{-\infty}^{+\infty} e^{\gamma x} \, d\mu_k(x) \right| \leq e^{c \gamma^2} \text{ for  all } \gamma \in \bR \text{ and for all } k \in \bZ^3,
 \end{equation}
and similarly for $\nu_k$. The assumption \eqref{eq:rvassumption} is satisfied, for example, by standard Gaussian random variables, standard Bernoulli random variables, or any random variables with compactly supported distributions. For a given $f = (f_1, f_2) \in H^s_x(\bR^3) \times H^{s-1}_x(\bR^3)$ we define its randomization by
 \begin{equation} \label{equ:bighsrandomization} 
  f^{\omega} = (f_1^{\omega}, f_2^{\omega}) := \biggl( \sum_{k \in \bZ^3} h_k(\omega) P_k f_1, \sum_{k \in \bZ^3} l_k(\omega) P_k f_2 \biggr).
 \end{equation}
The quantity $\sum_{k \in \bZ^3} h_k(\omega) P_k f_1$ is understood as the Cauchy limit in $L^2(\Omega; H^s_x(\bR^3))$ of the sequence $\Bigl( \sum_{|k| \leq N} h_k(\omega) P_k f_1 \Bigr)_{N \in \bN}$  and similarly for $\sum_{k \in \bZ^3} l_k(\omega) P_k f_2$. Let
 \begin{align} \label{eq:free_evolution}
  u_f^\omega = \cos(t|\nabla|) f_1^\omega + \frac{\sin(t|\nabla|)}{|\nabla|} f_2^\omega
 \end{align}
be the free wave evolution of the initial data $f^\omega$ defined in \eqref{equ:bighsrandomization}. 

\medskip

In the case of random variables such that there exists $c > 0$ for which their distributions satisfy
\begin{equation*}
\sup_{k\in \bZ^3} \mu_k([-c,c]) < 1,
\end{equation*}
one can show that if $f$ does not belong to $H^{s+\varepsilon}_x(\bR^3) \times H^{s-1+\varepsilon}_x(\bR^3)$, then the probability that $f^\omega$ belongs to $H^{s+\varepsilon}_x(\bR^3) \times H^{s-1+\varepsilon}_x(\bR^3)$ is zero. This follows by a similar argument to the one found in Appendix B of \cite{BT1}, using \eqref{eq:norm_equiv}. Thus, our randomization procedure does not regularize at the level of Sobolev spaces.

\medskip

We note that a similar randomization procedure on Euclidean space was used by Zhang and Fang in \cite[(1.12)]{ZF} to study random data local existence and small data global existence questions for the generalized incompressible Navier-Stokes equation on $\bR^d$ for $d \geq 3$.

\medskip

We are now in a position to state our results.
\begin{theorem} \label{thm:rho}
 Let $3 \leq \rho < 5$ and let 
 \begin{equation*}
  \frac{\rho^3+5\rho^2-11\rho-3}{9\rho^2-6\rho-3} < s < 1.
 \end{equation*}
 Let $f = (f_1, f_2) \in H^s_x(\bR^3) \times  H^{s-1}_x(\bR^3)$. Let $\{ (h_k, l_k) \}_{k \in \bZ^3}$  be a sequence of independent, zero-mean value, real-valued random variables on a probability space $(\Omega, {\mathcal A}, \bP)$ with distributions $\mu_k$ and $\nu_k$. Assume that there exists $c > 0$ such that
 \begin{equation*}
\left| \int_{-\infty}^{+\infty} e^{\gamma x} d\mu_k(x) \right| \leq e^{c \gamma^2} \text{ for  all } \gamma \in \bR \text{ and for all } k \in \bZ^3
 \end{equation*}
 and similarly for $\nu_k$. Let $f^\omega = (f_1^\omega, f_2^\omega)$ be the associated randomized initial data as defined in \eqref{equ:bighsrandomization} and let $u_f^\omega$ be the associated free evolution as defined in \eqref{eq:free_evolution}. For almost every $\omega \in \Omega$ there exists a unique global solution
 \begin{equation*}
  (u, u_t) \in (u_f^\omega, \partial_t u_f^\omega) + C\bigl(\bR; H^1_x(\bR^3) \times L^2_x(\bR^3)\bigr)
 \end{equation*}
 to the nonlinear wave equation
 \begin{equation} \label{equ:nlw_rho_global}
  \left\{ \begin{split}
   -u_{tt} + \Delta u &= |u|^{\rho-1} u \text{ on } \bR \times \bR^3, \\
   (u, u_t)|_{t=0} &= (f_1^\omega, f_2^\omega).
  \end{split} \right.
 \end{equation}
 Here, uniqueness only holds in a mild sense; see Remark \ref{remark:uniqueness}.
\end{theorem}
The proof of Theorem \ref{thm:rho} combines a probabilistic local existence argument with Bourgain's high-low frequency decomposition~\cite{B98}, an approach introduced by Colliander and Oh \cite[Theorem 2]{CO} to show almost sure global existence of the one-dimensional periodic defocusing cubic nonlinear Schr\"odinger equation below $L^2(\bT)$. The high-low method was previously used by Kenig, Ponce and Vega \cite[Theorem 1.2]{KPV} to prove global well-posedness of \eqref{equ:ivp} for $2 \leq \rho < 5$ for initial data in a range of sub-critical spaces below the energy space (see also \cite{BC}, \cite{GP} and \cite{Roy} for $\rho = 3$). We adopt arguments from \cite{KPV}. 

\medskip

Theorem~\ref{thm:rho} permits initial data at lower regularities than the deterministic result \cite[Theorem 1.2]{KPV}. This is mainly due to so-called averaging effects for the free evolution of randomized initial data (see Lemma \ref{lemma:averaging_effects_rho} below), which are proven by combining the unit-scale Bernstein estimate \eqref{equ:unit_scale_intro} and Strichartz estimates for the wave equation on $\bR^3$. In particular, we use here that the randomization is performed directly on Euclidean space.

\begin{remark}
For $3 \leq \rho < 4$, random data Cauchy theory for the nonlinear wave equation \eqref{equ:ivp} on $\bR^3$ has been addressed by Suzzoni \cite{Suzzoni1, Suzzoni2} using different approaches than in the proof of Theorem~\ref{thm:rho}. In \cite[Theorem 2]{Suzzoni1}, using methods from \cite{BT1, BT2}, almost sure global existence and scattering for \eqref{equ:ivp} for $3 \leq \rho < 4$ is established for radially symmetric initial data in a class of spaces of super-critical regularity related to the Penrose transform. For $\rho =3$, almost sure global existence and scattering for \eqref{equ:ivp} is proven in \cite[Theorem 1.2]{Suzzoni2}, using methods from \cite{BT4}. In both these cases, the spaces for the initial data do not coincide with $H^s_x(\bR^3) \times H^{s-1}_x(\bR^3)$. We do not address the question of scattering of the constructed solutions in Theorem 1.1. The main difficulty is that the high-low method does not yield bounds on any \emph{global} space-time norm of the nonlinear component of the solution $u$.
\end{remark}

\begin{remark}
During the final revision of this article, the preprints \cite{BOP2, BOP1} by B\'enyi-Oh-Pocovnicu and \cite{Pocovnicu} by Pocovnicu appeared which use a similar randomization procedure. In \cite{BOP2, BOP1} almost sure well-posedness results are established for the cubic nonlinear Schr\"odinger equation on $\mathbb{R}^d$ for $d \geq 3$. In \cite{Pocovnicu} almost sure global well-posedness is proven for the energy-critical defocusing nonlinear wave equation on $\mathbb{R}^d$ for $d = 4,5$.
\end{remark}

\begin{remark}
\label{rem:supercrit}
For $\frac{1}{4}(7+\sqrt{73}) \approx 3.89 < \rho < 5$, the range of allowable regularity exponents~$s$ in Theorem \ref{thm:rho} contains super-critical exponents; see Figure \ref{fig:plot}. To the best of the authors' knowledge, for $4 \leq \rho < 5$, Theorem~\ref{thm:rho} is the first result which establishes the existence of large sets of initial data of super-critical regularity which lead to global solutions to \eqref{equ:ivp}.
\end{remark}

\begin{figure}[hbt] 
 \centering
 \includegraphics[height=6cm,keepaspectratio]{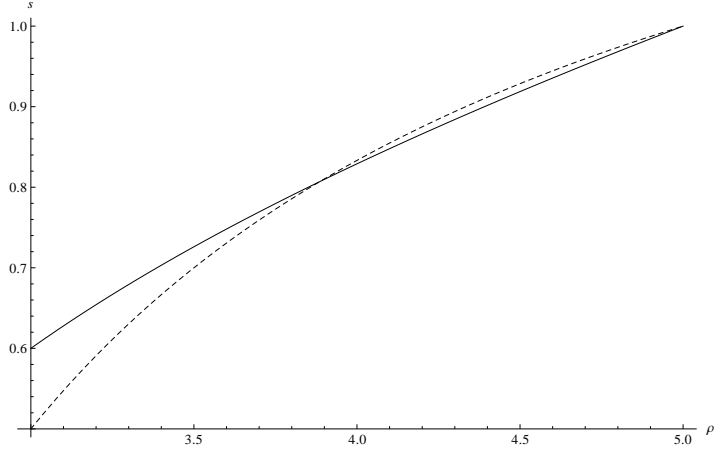}
 \label{fig:plot}
 \caption{The dashed line is the critical regularity $s_c = \frac{3}{2} - \frac{2}{\rho-1}$. The solid line is the threshold for the exponent $s$ in Theorem~\ref{thm:rho}.}
\end{figure} 

\begin{remark}
For the specific case $\rho = 3$, we note that our randomization procedure allows us to straightforwardly adapt the proofs in \cite{BT4} to the Euclidean setting. In particular, by modifying the proof of the averaging effects of Corollary A.4 of \cite{BT4}, we can use the arguments of Proposition 2.1 and Proposition 2.2 of \cite{BT4} to prove that if $0 < s < 1$ and $f = (f_1, f_2) \in H^s_x(\bR^3)\times H^{s-1}_x(\bR^3)$, then for almost every $\omega \in \Omega$ there exists a unique global solution
 \begin{equation*}
  (u, u_t) \in (u_f^\omega, \partial_t u_f^\omega) + C\bigl(\bR; H^1_x(\bR^3) \times L^2_x(\bR^3)\bigr)
 \end{equation*}
to the defocusing cubic nonlinear wave equation on $\bR^3$ with initial data $f^\omega$. This is an improvement over the corresponding range $\frac{3}{5} < s < 1$ in Theorem~\ref{thm:rho}. In contrast, these arguments which proceed via energy estimates cannot be applied directly in the case $3 < \rho < 5$. In such an argument, one would need bounds on $\| |u|^{\rho-1} u \|_{L_t^1 L_x^2(\bR \times \bR^3)} = \| u \|^{\rho}_{L_t^\rho L_x^{2\rho}(\bR \times \bR^3)}$ but one no longer has the necessary Sobolev embedding to close the argument.
\end{remark}

%\begin{remark}
%For the specific case $\rho = 3$, we note that our randomization procedure allows us to straightforwardly adapt the proofs of Proposition 2.1, Proposition 2.2 {\color{red}and Corollary A.4} in \cite{BT4} to the Euclidean setting. Namely, we are able to show that if $0 < s < 1$ and $f = (f_1, f_2) \in H^s_x(\bR^3)\times H^{s-1}_x(\bR^3)$, then for almost every $\omega \in \Omega$ there exists a unique global solution
% \begin{equation*}
%  (u, u_t) \in (u_f^\omega, \partial_t u_f^\omega) + C\bigl(\bR; H^1_x(\bR^3) \times L^2_x(\bR^3)\bigr)
% \end{equation*}
%to the defocusing cubic nonlinear wave equation on $\bR^3$ with initial data $f^\omega$. This is an improvement over the corresponding range $\frac{3}{5} < s < 1$ in Theorem~\ref{thm:rho}. In contrast, the argument of \cite{BT4} which proceeds via energy estimates cannot be applied directly in the case $3 < \rho < 5$. In such an argument, one would need bounds on $\| |u|^{\rho-1} u \|_{L_t^1 L_x^2(\bR \times \bR^3)} = \| u \|^{\rho}_{L_t^\rho L_x^{2\rho}(\bR \times \bR^3)}$ but one no longer has the necessary Sobolev embedding to close the argument.
%\end{remark}

When $\rho = 5$, the nonlinear wave equation \eqref{equ:ivp} is energy-critical and the high-low argument no longer works. However, by establishing probabilistic a priori estimates on the $L^5_t L^{10}_x(\bR \times \bR^3)$ norm of the free evolution $u_f^\omega$, we obtain the following  probabilistic small data global existence result.
%However, a similar argument to the one used to prove the key averaging effects for Theorem \ref{thm:rho} (see Lemma \ref{lemma:averaging_effects_rho} below) yields bounds on the $L^5_t L^{10}_x(\bR \times \bR^3)$ norm of the free evolution~$u_f^\omega$. Together with straightforward contraction and bootstrap arguments, we obtain the following theorem.

\begin{theorem} \label{thm:quintic}
 Let $\frac{2}{5} \leq s < 1$ and $f = (f_1, f_2) \in H^s_x(\bR^3) \times \dot{H}^{s-1}_x(\bR^3)$. Let $\{ (h_k, l_k) \}_{k \in \bZ^3}$  be a sequence of independent, zero-mean value, real-valued random variables on a probability space $(\Omega, {\mathcal A}, \bP)$ with distributions $\mu_k$ and $\nu_k$. Assume that there exists $c > 0$ such that
 \begin{equation*}
  \left| \int_{-\infty}^{+\infty} e^{\gamma x} d\mu_k(x) \right| \leq e^{c \gamma^2} \text{ for  all } \gamma \in \bR \text{ and for all } k \in \bZ^3
 \end{equation*}
 and similarly for $\nu_k$. Let $f^\omega = (f_1^\omega, f_2^\omega)$ be the associated randomized initial data as defined in \eqref{equ:bighsrandomization} and let $u_f^\omega$ be the associated free evolution as defined in \eqref{eq:free_evolution}. There exists $\Omega_f \subset \Omega$ with 
 \begin{equation} \label{equ:probabilistic_estimate_quintic_thm}
  \bP( \Omega_f ) \geq 1 - C e^{-c/\|f\|^2_{H^s_x \times \dot{H}^{s-1}_x}}  
 \end{equation}
 for some absolute constants $C, c > 0$, such that for every $\omega \in \Omega_f$ there exists a unique global solution 
 \begin{equation*}
  (u, u_t) \in (u_f^\omega, \partial_t u_f^\omega) + C\bigl(\bR; \dot{H}^1_x(\bR^3) \times L^2_x(\bR^3)\bigr)
 \end{equation*}
 to the quintic nonlinear wave equation
 \begin{equation} \label{equ:quintic_nlw}
  \left\{ \begin{split}
   -u_{tt} + \Delta u &= |u|^4 u \text{ on } \bR \times \bR^3, \\
   (u, u_t)|_{t = 0} &= (f_1^\omega, f_2^\omega).
  \end{split} \right.
 \end{equation}
 Moreover, we have scattering in the sense that for every $\omega \in \Omega_f$ there exist $(v_1, v_2) \in \dot{H}^1_x(\bR^3) \times L^2_x(\bR^3)$ such that the free evolution $v(t) = \cos(t |\nabla|) v_1 + \frac{\sin(t |\nabla|)}{|\nabla|} v_2$ satisfies
 \begin{equation*}
  \bigl\| \bigl( u(t) - u_f^\omega(t) - v(t), \partial_t u(t) - \partial_t u_f^\omega(t) - \partial_t v(t) \bigr) \bigr\|_{\dot{H}^1_x(\bR^3) \times L^2_x(\bR^3)} \longrightarrow 0 \text{ as } t \rightarrow \pm \infty.
 \end{equation*}
\end{theorem}

\begin{remark}
We note that the statement is only meaningful if
\begin{equation*}
\|f\|_{H^s_x \times \dot{H}^{s-1}_x} \leq \left( \frac{c}{\log(C)} \right)^{1/2},
\end{equation*}
which reflects that Theorem~\ref{thm:quintic} is a small data result.
\end{remark}

\noindent {\it Notation:} We let $C > 0$ denote a constant that depends only on fixed parameters and whose value may change from line to line. We write $X \lesssim Y$ to denote $X \leq C Y$ for some $C > 0$. As mentioned above, $H^s_x(\bR^3)$ denotes the usual inhomogeneous Sobolev space, while $\dot{H}^s_x(\bR^3)$ denotes the corresponding homogeneous space. Occasionally we will use the notation $\cH^s(\bR^3):= H^s_x(\bR^3) \times H^{s-1}_x(\bR^3)$, where $\cH^s(\bR^3)$ is endowed with the obvious norm.

\medskip

\noindent {\it Organization of the paper:} In Section \ref{preliminaries}, we present deterministic and probabilistic lemmata that will be used in the proof of Theorem \ref{thm:rho}. In Section \ref{sec:avg}, we prove certain probabilistic a priori estimates on the randomized initial data and the aforementioned averaging effects for the free evolution of the randomized initial data. Finally, in Section \ref{sec:nlw_existence}, we present the proof of Theorem \ref{thm:rho}, and in Section \ref{sec:quintic}, the proof of Theorem \ref{thm:quintic}.

\medskip

\noindent {\it Acknowledgments:} The authors would like to sincerely thank Gigliola Staffilani for all of her help. They are grateful to Michael Eichmair for his encouragement and support and to Andrea Nahmod for stimulating discussions. The first author wishes to thank his advisor Michael Struwe for his support and guidance.

\section{Deterministic and probabilistic preliminaries} \label{preliminaries}
\setcounter{equation}{0}

\subsection{Deterministic preliminaries}

The following unit-scale Bernstein estimate is a crucial ingredient in the proofs of the probabilistic a priori estimates on the randomized initial data in Section \ref{sec:avg}. Its advantage compared to the ordinary Bernstein estimate for the dyadic Littlewood-Paley decomposition is that there is no derivative loss.

\begin{lemma}[``Unit-scale Bernstein estimate'']
\label{lem:unit_bernstein}
 Let $2 \leq p_1 \leq p_2 \leq \infty$. There exists a constant $C \equiv C(p_1, p_2) > 0$ such that for all $f \in L^2_x(\bR^3)$ and for all $k \in \bZ^3$
 \begin{equation} \label{equ:bernstein_estimate}
  \|P_k f\|_{L^{p_2}_x(\bR^3)} \leq C \|P_k f\|_{L^{p_1}_x(\bR^3)}.
 \end{equation}
\end{lemma}
\begin{proof}
 Let $\eta \in C_c^{\infty}(\bR^3)$ be such that $0 \leq \eta \leq 1$ with $\eta(\xi) = 1$ for $|\xi| \leq 2$ and $\eta(\xi) = 0$ for $|\xi| \geq 3$. For $k \in \bZ^3$ define 
 \begin{equation*}
  \eta_k(\xi) = \eta(\xi - k) \text{ for } \xi \in \bR^3.
 \end{equation*}
 From Young's inequality with $1 + \frac{1}{p_2} = \frac{1}{q} + \frac{1}{p_1}$ we then obtain
 \begin{equation*}
  \begin{split}
  \|P_k f\|_{L^{p_2}_x(\bR^3)} &= \|\cF^{-1} ( \psi_k \hat{f} )\|_{L^{p_2}_x(\bR^3)} = \|\cF^{-1} ( \eta_k \psi_k \hat{f} )\|_{L^{p_2}_x(\bR^3)} =  \|\check{\eta}_k  \ast (P_k f)\|_{L^{p_2}_x(\bR^3)} \\
  &\leq \|\check{\eta}_k\|_{L^q_x(\bR^3)} \|P_k f\|_{L^{p_1}_x(\bR^3)} = \|\check{\eta}\|_{L^q_x(\bR^3)} \|P_k f\|_{L^{p_1}_x(\bR^3)}. \qedhere
  \end{split} 
 \end{equation*}
\end{proof}

In the next sections we will repeatedly make use of Strichartz estimates for the wave equation on $\bR^3$.

\begin{definition}
 We call an exponent pair $(q,r)$ wave-admissible if $2 \leq q \leq \infty$, $2 \leq r < \infty$ and
 \begin{equation}
 \label{equ:wave}
  \frac{1}{q} + \frac{1}{r} \leq \frac{1}{2}.
 \end{equation}
\end{definition}

\begin{proposition}[\protect{Strichartz estimates in three space dimensions, \cite{Strichartz}, \cite{Pecher}, \cite{GV2}, \cite{KeelTao}}] \label{prop:strichartz_estimates}
 Suppose $(q, r)$ and $(\tilde{q}, \tilde{r})$ are wave-admissible pairs. Let $u$ be a (weak) solution to the wave equation
 \begin{equation*}
  \left\{ \begin{split}
   -u_{tt} + \Delta u \, &= \, h \text{ on } [0, T] \times \bR^3, \\
   (u, u_t)|_{t=0} \, &= \, (f,g)
  \end{split} \right.
 \end{equation*}
 for some data $f, g, h$ and time $0 < T < \infty$. Then
 \begin{equation} \label{equ:strichartz_estimates}
  \begin{split}
   &\|u\|_{L^q_t L^r_x ([0, T] \times \bR^3)} + \|u\|_{L^{\infty}_t \dot{H}^{\gamma}_x([0,T] \times \bR^3)} + \|u_t\|_{L^{\infty}_t \dot{H}^{\gamma-1}_x([0,T] \times \bR^3)} \\
   &\qquad \qquad \lesssim \, \|f\|_{\dot{H}_x^{\gamma}(\bR^3)} + \|g\|_{\dot{H}_x^{\gamma-1}(\bR^3)} + \|h\|_{L_t^{\tilde{q}'} L^{\tilde{r}'}_x([0,T] \times \bR^3)}
  \end{split}
 \end{equation}
 under the assumption that the scaling conditions
 \begin{equation}
 \label{equ:strichartz1}
\qquad  \frac{1}{q} + \frac{3}{r} \, = \, \frac{3}{2} - \gamma
 \end{equation}
 and
 \begin{equation}
 \label{equ:strichartz2}
\frac{1}{\tilde{q}'} + \frac{3}{\tilde{r}'} - 2 \, = \, \frac{3}{2} - \gamma
 \end{equation} 
 hold.
\end{proposition}

We say that a wave-admissible pair $(q,r)$ is Strichartz-admissible at regularity $\gamma$ if it satisfies \eqref{equ:strichartz1} for some $0 < \gamma < \frac{3}{2}$.

\subsection{Probabilistic preliminaries}

First, we recall a large deviation estimate from \cite{BT1} that will be used to obtain the probabilistic a priori estimates on the randomized initial data in Section \ref{sec:avg}.

\begin{proposition}[\protect{Large deviation estimate, \cite[Lemma 3.1]{BT1}}] \label{prop:large_deviation_estimate}
 Let $\{l_n\}_{n=1}^{\infty}$ be a sequence of real-valued independent random variables with associated distributions $\{\mu_n\}_{n=1}^{\infty}$ on a probability space $(\Omega, {\mathcal A}, \bP)$. Assume that the distributions satisfy the property that there exists $c > 0$ such that
 \begin{equation*}
  \left| \int_{-\infty}^{+\infty} e^{\gamma x} d\mu_n(x) \right| \leq e^{c \gamma^2} \text{ for  all } \gamma \in \bR \text{ and for all } n \in \mathbb{N}.
 \end{equation*}
 Then there exists $C > 0$ such that for every $p \geq 2$ and every $\{c_n\}_{n=1}^{\infty} \in \ell^2(\bN; \bC)$,
 \begin{equation} \label{equ:large_deviation_estimate}
  \Bigl\| \sum_{n=1}^{\infty} c_n l_n(\omega) \Bigr\|_{L^p_\omega(\Omega)} \leq C \sqrt{p} \Bigl( \sum_{n=1}^{\infty} |c_n|^2 \Bigr)^{1/2}.
 \end{equation}
\end{proposition}

The next lemma will be used to estimate the probability of certain subsets of the probability space $\Omega$. Its proof is a straightforward adaption of the proof of Lemma 4.5 in \cite{Tz10}.

\begin{lemma} \label{lem:tzvetkov_probest}
Let $f \in \cH^s(\bR^3)$ and let $F$ be a real-valued measurable function on a probability space $(\Omega, {\mathcal A}, \bP)$. Suppose that there exist $C > 0$ and $r \geq 1$ such that for every $p\geq r$ we have
\begin{align*}
\|F\|_{L^p_{\omega}(\Omega)} \leq C \sqrt{p} \|f\|_{\cH^s}. 
\end{align*}
Then there exist $\delta > 0$ and $C_1 > 0$, depending on $C$ and $r$ but independent of $f$, such that for every $\lambda > 0$,
\begin{align*}
  \bP( \omega \in \Omega : |F(\omega)| > \lambda ) \leq C_1 e^{-\delta \lambda^{2}/\|f\|^2_{\cH^s}}.
\end{align*}
\end{lemma}

\section{Averaging effects for the randomized initial data} \label{sec:avg}
\setcounter{equation}{0}

In the proof of Theorem \ref{thm:rho} we decompose the randomized initial data into a low frequency and a high frequency component. More precisely, for $N \in \bN$ define
\begin{equation*}
 f_{1, \leq N}^\omega = \sum_{|k| \leq N} h_k(\omega) P_k f_1
\end{equation*}
and $f_{1, >N}^\omega = f_1^\omega - f_{1, \leq N}^\omega$. Similarly, define $f_{2, \leq N}^\omega$ and $f_{2, >N}^\omega$. The next two lemmata establish probabilistic a priori estimates on the low frequency component. See also Proposition 4.4 in \cite{BT1}.

\begin{lemma}
 Let $3 \leq \rho < 5$. Let $s \geq 0$ and $f = (f_1, f_2) \in \cH^s(\bR^3)$. For any $K > 0$ and $N \in \bN$, let
 \begin{equation} \label{equ:A_K,N}
  A_{K, N} := \bigl\{ \omega \in \Omega : \|f_{1, \leq N}^\omega\|^{(\rho+1)/2}_{L^{\rho+1}_x(\bR^3)} \leq K \bigr\}.
 \end{equation}
 Then there exist constants $C \equiv C(\rho) > 0$ and $c \equiv c(\rho) > 0$ such that
 \begin{equation} \label{equ:prob_estimate_A_KN}
  \bP(A_{K,N}^c) \leq C e^{- c K^{4/(\rho+1)}/\|f\|^2_{\cH^s(\bR^3)}} 
 \end{equation}
  for every $K > 0$ and $N \in \bN$.
\end{lemma}
\begin{proof}
 For every $p \geq \rho+1$ and every $N \in \bN$, we bound 
 \begin{align}
  \|f_{1, \leq N}^\omega\|_{L^p_\omega(\Omega; L^{\rho+1}_x(\bR^3))} &= \Bigl\| \sum_{|k| \leq N} h_k(\omega) P_k f_1 \Bigr\|_{L^p_\omega(\Omega; L^{\rho + 1}_x(\bR^3))} \\
  &\leq \Bigl\| \bigl\| \sum_{|k| \leq N} h_k(\omega) P_k f_1  \bigr\|_{L^p_\omega(\Omega)}  \Big\|_{L^{\rho+1}_x(\bR^3)} \label{equ:averaging_L_rho_1} \\
  &\lesssim \sqrt{p} \, \Bigl\| \bigl( \sum_{|k| \leq N} |P_k f_1(x)|^2  \bigr)^{1/2} \Big\|_{L^{\rho+1}_x(\bR^3)} \label{equ:averaging_L_rho_2}\\
  &\lesssim \sqrt{p} \, \Bigl( \sum_{|k| \leq N} \|P_k f_1\|_{L^{\rho+1}_x(\bR^3)}^2  \Bigr)^{1/2} \\
  &\lesssim \sqrt{p} \, \Bigl( \sum_{|k| \leq N} \|P_k f_1\|_{L^{2}_x(\bR^3)}^2  \Bigr)^{1/2} \label{equ:averaging_L_rho_3}\\ 
  &\lesssim \sqrt{p} \, \|f_1\|_{L^2_x(\bR^3)}. 
 \end{align}
Using $p \geq \rho +1$, we switched the order of integration in \eqref{equ:averaging_L_rho_1} and used the large deviation estimate~\eqref{equ:large_deviation_estimate} in \eqref{equ:averaging_L_rho_2}. We then used the unit-scale Bernstein estimate~\eqref{equ:bernstein_estimate} in \eqref{equ:averaging_L_rho_3}. The claim follows from Lemma \ref{lem:tzvetkov_probest}. 
\end{proof}

\begin{lemma}
 Let $s \in \bR$ and $f = (f_1, f_2) \in \cH^s(\bR^3)$. For any $K > 0$ and $N \in \bN$, let
 \begin{equation} \label{equ:B_K,N}
  B_{K, N} := \bigl\{ \omega \in \Omega : \|f_{1, \leq N}^\omega\|_{H^s_x(\bR^3)} + \|f_{2, \leq N}^\omega\|_{H^{s-1}_x(\bR^3)} \leq K \bigr\}.
 \end{equation}
 Then there exist constants $C \equiv C(s) > 0$ and  $c \equiv c(s) > 0$ such that
 \begin{equation} \label{equ:prob_estimate_B_KN}
  \bP(B_{K,N}^c) \leq C e^{- c K^2/\|f\|^2_{\cH^s(\bR^3)}}
 \end{equation}
 for every $K > 0$ and $N \in \bN$.
\end{lemma}
\begin{proof}
 For every $p \geq 2$ and every $N \in \bN$, using the large deviation estimate \eqref{equ:large_deviation_estimate} we have
 \begin{equation*}
  \begin{split}
   \bigl\| (1-\Delta)^{s/2} f_{1, \leq N}^\omega \bigr\|_{L^p_\omega(\Omega; L^2_x(\bR^3))} &= \Bigl\| \sum_{|k| \leq N} h_k(\omega) ( (1-\Delta)^{s/2} P_k f_1 ) \Bigr\|_{L^p_\omega(\Omega; L^2_x(\bR^3))} \\
   &\lesssim \sqrt{p} \Bigl\| \Bigl( \sum_{|k|\leq N} \bigl|( (1-\Delta)^{s/2} P_k f_1)(x)\bigr|^2 \Bigr)^{1/2} \Bigr\|_{L^2_x(\bR^3)} \\
   &\lesssim \sqrt{p} \Bigl( \sum_{|k| \leq N} \bigl\| (1-\Delta)^{s/2} P_k f_1 \bigr\|^2_{L^2_x(\bR^3)} \Bigr)^{1/2} \\
   &\lesssim \sqrt{p} \|f_1\|_{H^s_x(\bR^3)}.
  \end{split}
 \end{equation*}
 Similarly, we obtain
 \begin{equation*}
  \bigl\| (1-\Delta)^{(s-1)/2} f_{2, \leq N}^\omega \bigr\|_{L^p_\omega(\Omega; L^2_x(\bR^3))} \lesssim \sqrt{p} \|f_2\|_{H^{s-1}_x(\bR^3)}.
 \end{equation*}
 The assertion then follows from Lemma \ref{lem:tzvetkov_probest}.
\end{proof}

Analogously to \eqref{eq:free_evolution}, define the free evolution of the high frequency component of the randomized initial data by
\begin{equation*}
 u_{f, >N}^\omega =  \cos(t|\nabla|) f_{1, >N}^\omega + \frac{\sin(t|\nabla|)}{|\nabla|} f_{2,>N}^\omega.
\end{equation*}
In the next lemma we prove probabilistic a priori estimates which exhibit a decay in $N$ on certain space-time norms of the free evolution $u_{f, >N}^\omega$ once one restricts to suitable subsets of the probability space $\Omega$. The decay is ultimately the reason why we obtain an improved range of regularity exponents $s$ in Theorem \ref{thm:rho} compared to \cite[Theorem~1.2]{KPV} and results from the use of the unit-scale Bernstein estimate \eqref{equ:bernstein_estimate} and the Strichartz estimate \eqref{equ:strichartz_estimates}.

%The averaging effects for the free evolution $u_{f, >N}^\omega$ in Lemma~\ref{lemma:averaging_effects_rho} are a key element in the proof of Theorem~\ref{thm:rho}. Note that these bounds yield a decay in $N$ of certain space-time norms of $u_{f, >N}^\omega$ once one restricts to suitable subsets of the probability space $\Omega$. In particular, although we will use similar arguments as in Kenig-Ponce-Vega \cite{KPV}, this decay is the reason why we obtain an improved range of regularity exponents~$s$ in Theorem~\ref{thm:rho} compared to \cite[Theorem~1.2]{KPV}.
\begin{lemma} \label{lemma:averaging_effects_rho}
 Let $T > 0$ and $3 \leq \rho < 5$. Let $0 < s < 1$ and $0 < \varepsilon < \min(\frac{s}{2}, \frac{1}{2}(1-\frac{1}{\rho}))$. Let $f = (f_1, f_2) \in \cH^s(\bR^3)$. For any $K > 0$ and $N \geq 3$, let
 \begin{equation} \label{equ:D_K,N}
  D_{K, N, \,\varepsilon} := \bigl\{ \omega \in \Omega : N^{s - 2\varepsilon} \|u_{f, >N}^\omega\|_{L^{1/\varepsilon}_t L^{2\rho}_x([0,T]\times\bR^3)} \leq K \bigr\}.
 \end{equation}
 Then there exist constants $C \equiv C(s, \rho, \varepsilon) > 0$ and $c \equiv c(s, \rho, \varepsilon) > 0$ such that
 \begin{equation} \label{equ:prob_estimate_D_KN}
  \bP(D_{K,N, \,\varepsilon}^c) \leq C e^{- c K^2/\|f\|^2_{\cH^s(\bR^3)}} 
 \end{equation}
 for every $K > 0$ and $N \geq 3$.
\end{lemma}
\begin{proof}
 Set $r(\varepsilon) = \frac{2}{1-2\varepsilon}$. Then the exponent pair $(\frac{1}{\varepsilon}, r(\varepsilon))$ is wave-admissible and Strichartz-admissible at regularity $\gamma = 2 \varepsilon$. For every $p \geq \max(\frac{1}{\varepsilon}, 2\rho)$ and every $N \geq 3$, we now~estimate
 \begin{align}
   &\bigl\| u_{f, >N}^\omega \bigr\|_{L^p_\omega(\Omega; L^{1/\varepsilon}_t L^{2 \rho}_x([0,T]\times\bR^3))} \notag \\
   &\lesssim \sqrt{p} \Bigl( \sum_{|k| > N} \bigl\| \cos(t|\nabla|) P_k f_1 \bigr\|^2_{L^{1/\varepsilon}_t L^{2\rho}_x([0,T]\times\bR^3)} \Bigr)^{1/2}  \label{eq:averaging_rho_1} \\
   &\quad \quad + \sqrt{p} \Bigl( \sum_{|k| > N} \Bigl\| \frac{\sin(t|\nabla|)}{|\nabla|}  P_k f_2 \Bigr\|^2_{L^{1/\varepsilon}_t L^{2\rho}_x([0,T]\times\bR^3)} \Bigr)^{1/2} \notag \\
   &\lesssim \sqrt{p} \Bigl( \sum_{|k| > N} \bigl\| \cos(t|\nabla|) P_k f_1 \bigr\|^2_{L^{1/\varepsilon}_t L^{r(\varepsilon)}_x([0,T]\times\bR^3)} \Bigr)^{1/2} \label{eq:averaging_rho_2} \\
   &\quad \quad + \sqrt{p} \Bigl( \sum_{|k| > N} \Bigl\| \frac{\sin(t|\nabla|)}{|\nabla|} P_k f_2 \Bigr\|^2_{L^{1/\varepsilon}_t L^{r(\varepsilon)}_x([0,T]\times\bR^3)} \Bigr)^{1/2}  \notag \\
   &\lesssim \sqrt{p} \Bigl( \sum_{|k| > N} \bigl\| P_k f_1 \bigr\|^2_{\dot{H}^{2\varepsilon}_x(\bR^3)} \Bigr)^{1/2} + \sqrt{p} \Bigl( \sum_{|k| > N} \bigl\| P_k f_2 \bigr\|^2_{\dot{H}^{2\varepsilon -1}_x(\bR^3)} \Bigr)^{1/2} \label{eq:averaging_rho_3} \\
   &\lesssim \sqrt{p} N^{-(s-2\varepsilon)} \Bigl( \sum_{|k| > N} \bigl\| P_k f_1 \bigr\|^2_{\dot{H}^{s}_x(\bR^3)} \Bigr)^{1/2} + \sqrt{p} N^{-(s-2\varepsilon)} \Bigl( \sum_{|k| > N} \bigl\| P_k f_2 \bigr\|^2_{\dot{H}^{s -1}_x(\bR^3)} \Bigr)^{1/2} \\
   &\lesssim \sqrt{p} N^{-(s-2\varepsilon)} \Bigl( \sum_{|k| > N} \bigl\| P_k f_1 \bigr\|^2_{H^s_x(\bR^3)} \Bigr)^{1/2} + \sqrt{p} N^{-(s-2\varepsilon)} \Bigl( \sum_{|k| > N} \bigl\| P_k f_2 \bigr\|^2_{H^{s-1}_x(\bR^3)} \Bigr)^{1/2} \label{eq:averaging_rho_4} \\ 
   &\lesssim \sqrt{p} N^{-(s-2\varepsilon)} \|f\|_{\cH^s(\bR^3)}. \notag
 \end{align}
Using $p \geq \max(\frac{1}{\varepsilon}, 2\rho)$, we switched the order of integration and used the large deviation estimate \eqref{equ:large_deviation_estimate} in \eqref{eq:averaging_rho_1}. The assumption $f \in \cH^s(\bR^3)$ together with Plancherel's theorem guarantees that 
 \begin{equation*}
 (\sum_{k \in \bZ^3} |(\cos(t|\nabla|) P_k f_1)(x)|^2 )^{1/2} < \infty \qquad \textup{and} \qquad \Bigl(\sum_{k \in \bZ^3} \Bigl|\Bigl(\frac{\sin(t|\nabla|)}{|\nabla|} P_k f_2 \Bigr)(x) \Bigr|^2 \Bigr)^{1/2} < \infty
 \end{equation*}
 for almost every $x \in \bR^3$ and for every $t \in [0,T]$, allowing us to apply \eqref{equ:large_deviation_estimate}. We use the unit-scale Bernstein estimate \eqref{equ:bernstein_estimate} in \eqref{eq:averaging_rho_2}, noting that $r(\varepsilon) \leq 2 \rho$, and then apply the Strichartz estimates \eqref{equ:strichartz_estimates} at regularity $\gamma = 2 \varepsilon$ in \eqref{eq:averaging_rho_3}. In  \eqref{eq:averaging_rho_4} we may estimate $\|P_k f_2 \|_{\dot{H}^{s-1}_x(\bR^3)} \lesssim \| P_k f_2 \|_{H^{s-1}_x(\bR^3)}$ uniformly for all $|k| > N \geq 3$ even though $s-1 < 0$, since $\cF(P_k f_2)(\xi) = 0$ for $|\xi| < 1$ for all $|k| > N \geq 3$ due to the support properties of the unit-scale projections. The claim then follows from Lemma \ref{lem:tzvetkov_probest}.
\end{proof}

\section{Proof of Theorem \ref{thm:rho}} \label{sec:nlw_existence}
\setcounter{equation}{0}

This section is devoted to the proof of the following proposition, which immediately implies Theorem~\ref{thm:rho}.
\begin{proposition} \label{prop:rho}
 Let $T > 0$. Let $3 \leq \rho < 5$ and 
 \begin{equation*}
  \frac{\rho^3+5\rho^2-11\rho-3}{9\rho^2-6\rho-3} < s < 1.
 \end{equation*}
 Fix $f = (f_1, f_2) \in H^s_x(\bR^3) \times H^{s-1}_x(\bR^3)$. Let $f^\omega = (f_1^\omega, f_2^\omega)$ be the associated randomized initial data as defined in~\eqref{equ:bighsrandomization} and $u_f^\omega$ the corresponding free evolution as defined in~\eqref{eq:free_evolution}. Then there exists $\Omega_T \subset \Omega$ with $\bP(\Omega_T) = 1$ such that for every $\omega \in \Omega_T$ there exists a unique solution
 \begin{equation} \label{equ:solution_form_rho}
  (u, u_t) \in (u_f^\omega, \partial_t u_f^\omega) + C\bigl([0,T]; H^1_x(\bR^3) \times L^2_x(\bR^3)\bigr)
 \end{equation}
 to the nonlinear wave equation
 \begin{equation} \label{equ:nlw_rho}
\left\{  \begin{split}
   -u_{tt} + \Delta u &= |u|^{\rho-1} u \text{ on } [0,T] \times \bR^3, \\
   (u, u_t)|_{t=0} &= (f_1^\omega, f_2^\omega).
  \end{split} \right.
 \end{equation}
 Here, uniqueness only holds in a mild sense; see Remark \ref{remark:uniqueness}.
\end{proposition}

\begin{proof}[Proof of Theorem \ref{thm:rho}]
We only present the argument to construct solutions that exist for all positive times since the argument for negative times is similar. Define $T_j = j$ for $j \in \bN$ and set $\Sigma = \bigcap_{j=1}^\infty \Omega_{T_j}$. By Proposition~\ref{prop:rho}, we have that $\bP(\Sigma) = 1$ and for every $\omega \in \Sigma$, we have global existence for \eqref{equ:nlw_rho_global} on the time interval $[0, \infty)$.
\end{proof}

In the proof of Proposition \ref{prop:rho} we will repeatedly use the following probabilistic low regularity local well-posedness result whose proof we defer to the end of this section. It is here that we invoke the crucial averaging effects for the free evolution $u_{f, >N}^\omega$ from Lemma~\ref{lemma:averaging_effects_rho}. We introduce the notation
\begin{align*}
 q(\rho) = \frac{2 \rho}{\rho - 3}, \quad \alpha(\rho) = \frac{5 - \rho}{2}.
\end{align*}

\begin{lemma} \label{lemma:probabilistic_lwp}
Let $3 \leq \rho < 5$ and $0 < s < 1.$ Let $0 < \varepsilon < \min(\frac{s}{2}, \frac{1}{2}(1-\frac{1}{\rho}))$ and $K > 0$ be fixed. For $N \geq 3$ and $0 < c < 1$ set $T_1 = c (K N^{1-s} )^{-(\rho - 1)/\alpha(\rho)}$. Let $v: [0,T_1]\times\bR^3 \rightarrow \bC$ satisfy
 \begin{equation} \label{equ:lwp_v_bound}
  \|v\|_{L^{q(\rho)}_t L^{2 \rho}_x([0,T_1]\times\bR^3)} \leq C K N^{1-s}
 \end{equation}
 and let $\omega \in \Omega$ be such that
 \begin{equation} \label{equ:lwp_u_bound}
  \bigl\| u_{f, >N}^\omega \bigr\|_{L^{1/\varepsilon}_t L^{2\rho}_x([0, T_1]\times\bR^3)} \leq K N^{-s + 2 \varepsilon}.
 \end{equation}
 For $0 < c < 1$ sufficiently small (independent of the size of $K$ and $N$) and $N \equiv N(K)$ sufficiently large, there exists a unique solution
 \begin{equation*}
  (\widetilde{w}, \widetilde{w}_t) \in C\bigl([0,T_1]; \dot{H}^1_x(\bR^3)\bigr) \cap L^{q(\rho)}_t L^{2 \rho}_x\bigl([0,T_1]\times\bR^3\bigr) \times C\bigl([0,T_1]; L^2_x(\bR^3)\bigr)
 \end{equation*}
 to the nonlinear wave equation
 \begin{equation} \label{equ:lwp_nlw}
  \left\{\begin{split}
    -\widetilde{w}_{tt} + \Delta \widetilde{w} &= \bigl| v + u_{f, >N}^\omega + \widetilde{w} \bigr|^{\rho -1} \bigl( v + u_{f, >N}^\omega + \widetilde{w} \bigr) - |v|^{\rho -1} v \text{ on } [0,T_1] \times \bR^3, \\
    (\widetilde{w}, \widetilde{w}_t)|_{t = 0} &= (0,0),
  \end{split} \right.
 \end{equation}
 satisfying
 \begin{equation} \label{equ:lwp_w_bounds}
  \begin{split}
  \|\widetilde{w}(T_1)\|_{\dot{H}^1_x(\bR^3)} + \|\widetilde{w}_t(T_1)\|_{L^2_x(\bR^3)} + \|\widetilde{w}(T_1)\|_{L^{\rho + 1}_x(\bR^3)} & \lesssim T_1^{1 - \frac{\rho-1}{q(\rho)} -\varepsilon}  K^\rho N^{2\varepsilon + (1-s)\rho-1}.
  \end{split}
 \end{equation}
\end{lemma}

We now present the proof of Proposition \ref{prop:rho}.

\begin{proof}[Proof of Proposition \ref{prop:rho}]
The bulk of the proof is devoted to the construction of subsets $\Omega_{K,T} \subset \Omega$ for every $K \in \bN$ such that for every $\omega \in \Omega_{K,T}$ there exists a unique solution of the form \eqref{equ:solution_form_rho} to \eqref{equ:nlw_rho} and such that
\begin{equation} \label{equ:probability_estimate_Omega_K}
 \bP(\Omega_{K,T}^c) \leq C e^{- C K^{ 4/(\rho +1)} / \|f\|^2_{\cH^s(\bR^3)} }.
\end{equation}
We then set $\Omega_T = \cup_{K=1}^{\infty} \Omega_{K,T}$ and conclude from \eqref{equ:probability_estimate_Omega_K} that $\bP(\Omega_T) = 1$, which completes the proof of Proposition~\ref{prop:rho}.

\medskip

In what follows, let $K \in \bN$ be fixed. Let $N \equiv N(K,T) \in \bN$ be sufficiently large, to be fixed later in the proof. We also make use of a fixed small parameter $0 < \varepsilon < \min(\frac{s}{2}, \frac{1}{2}(1-\frac{1}{\rho}))$ whose value depends on $\rho$ and $s$, but is independent of $K$ and $N$, and is specified further below.
We define 
\begin{equation*}
 \Omega_{K, T} = A_{K, N} \cap B_{K, N} \cap D_{K, N, \,\varepsilon},
\end{equation*}
where these sets are as in \eqref{equ:A_K,N}, \eqref{equ:B_K,N} and \eqref{equ:D_K,N}. The estimate \eqref{equ:probability_estimate_Omega_K} then follows from \eqref{equ:prob_estimate_A_KN}, \eqref{equ:prob_estimate_B_KN} and \eqref{equ:prob_estimate_D_KN}.  

From now on we only consider $\omega \in \Omega_{K,T}$. It suffices to show that there exists a unique solution
\begin{equation}
  (u, u_t) \in (u_f^\omega, \partial_t u_f^\omega) + C\bigl([0,T]; \dot H^1_x(\bR^3) \times L^2_x(\bR^3)\bigr),
\end{equation}
since by a persistence of regularity argument, one has $u \in u_f^\omega + C\bigl([0,T];  H^1_x(\bR^3)\bigr)$.

We first construct a solution $u^{(1)} = v^{(1)} + w^{(1)}$ to \eqref{equ:nlw_rho} on a small time interval $[0,T_1]$ with $0 < T_1 < 1$ to be fixed later and where $v^{(1)}$ solves the following nonlinear wave equation with low frequency initial data
\begin{equation} \label{equ:nlw_for_v_first_iteration}
\left\{ \begin{split}
   -v^{(1)}_{tt} + \Delta v^{(1)} &= |v^{(1)}|^{\rho-1} v^{(1)} \text{ on } [0,T_1] \times \bR^3, \\
   (v^{(1)}, v^{(1)}_t)|_{t=0} &= (f_{1, \leq N}^\omega, f_{2, \leq N}^\omega). 
 \end{split} \right.
\end{equation}
Note that the initial data $(f_{1, \leq N}^\omega, f_{2, \leq N}^\omega)$ lies in $\dot{H}^1_x(\bR^3) \times L^2_x(\bR^3)$, since by \eqref{equ:B_K,N}
\begin{equation} \label{equ:bound_kinetic_energy_low_frequency_data}
 \|f_{1, \leq N}^\omega\|_{\dot{H}^1_x(\bR^3)} + \|f_{2, \leq N}^\omega\|_{L^2_x(\bR^3)} \lesssim N^{1-s} (\|f_{1, \leq N}^\omega\|_{H^s_x(\bR^3)} + \|f_{2, \leq N}^\omega\|_{H^{s-1}_x(\bR^3)} ) \lesssim K N^{1-s}. 
\end{equation}
Thus, by the deterministic global existence theory \cite[Proposition 3.2]{GV} there exists a unique solution $(v^{(1)}, v^{(1)}_t) \in C\bigl([0,T_1]; \dot{H}^1_x(\bR^3) \times L^2_x(\bR^3)\bigr)$ to \eqref{equ:nlw_for_v_first_iteration}. Moreover, we have energy conservation since $\|f_{1, \leq N}^\omega\|^{\rho +1}_{L^{\rho + 1}_x(\bR^3)} \leq K^2$ by \eqref{equ:A_K,N}. Hence, for all $t \in [0, T_1]$
\begin{equation} \label{equ:energy_conservation}
 \begin{split}
  E(v^{(1)}(t)) &:= \frac{1}{2} \|v^{(1)}(t)\|^2_{\dot{H}^1_x(\bR^3)} + \frac{1}{2} \|v^{(1)}_t(t)\|^2_{L^2_x(\bR^3)} + \frac{1}{\rho +1} \|v^{(1)}(t)\|^{\rho +1}_{L^{\rho +1}_x(\bR^3)} \\
  &= \frac{1}{2} \|f_{1, \leq N}^\omega\|^2_{\dot{H}^1_x(\bR^3)} + \frac{1}{2} \|f_{2, \leq N}^\omega\|^2_{L^2_x(\bR^3)} + \frac{1}{\rho +1} \|f_{1, \leq N}^\omega\|^{\rho+1}_{L^{\rho +1}_x(\bR^3)} \\
  &\lesssim ( K N^{1-s} )^2.
 \end{split}
\end{equation}
We note that the exponent pair $(q(\rho), 2 \rho)$ is Strichartz-admissible at regularity $\gamma = 1$. Using Strichartz estimates \eqref{equ:strichartz_estimates} and \eqref{equ:bound_kinetic_energy_low_frequency_data}, we find
\begin{equation*}
 \begin{split}
  \|v^{(1)}\|_{L^{q(\rho)}_t L^{2\rho}_x([0,T_1]\times\bR^3)} &\lesssim \|v^{(1)}(0)\|_{\dot{H}^1_x(\bR^3)} + \|v^{(1)}_t(0)\|_{L^2_x(\bR^3)} + \bigl\| |v^{(1)}|^{\rho -1} v^{(1)} \bigr\|_{L^1_t L^2_x([0,T_1]\times\bR^3)} \\
  &\lesssim K N^{1-s} + T_1^{\alpha(\rho)} \|v^{(1)}\|^\rho_{L^{q(\rho)}_t L^{2 \rho}_x([0,T_1]\times\bR^3)}.
 \end{split}
\end{equation*}
Hence, choosing $T_1 = c ( K N^{1-s} )^{-\frac{\rho-1}{\alpha(\rho)}}$ with $0 < c < 1$ sufficiently small (independently of the size of $K$ and $N$), we obtain 
\begin{equation} \label{equ:strichartz_bounds_on_v}
 \|v^{(1)}\|_{L^{q(\rho)}_t L^{2 \rho}_x([0,T_1]\times\bR^3)} \lesssim K N^{1-s}.
\end{equation}

Next, we consider the nonlinear wave equation that $w^{(1)} = u^{(1)} - v^{(1)}$ must satisfy, namely
\begin{equation} \label{equ:nlw_for_w_first_iteration}
\left\{ \begin{split}
   -w^{(1)}_{tt} + \Delta w^{(1)} &= \bigl|v^{(1)} + w^{(1)}\bigr|^{\rho-1} \bigl(v^{(1)} + w^{(1)}\bigr) - \bigl| v^{(1)} \bigr|^{\rho-1} v^{(1)} \text{ on } [0,T_1] \times \bR^3, \\
   (w^{(1)}, w^{(1)}_t)|_{t=0} &= (f_{1, > N}^\omega, f_{2, > N}^\omega). 
 \end{split} \right.
\end{equation}
We look for a solution of the form
\begin{equation*}
 w^{(1)} = u_{f, > N}^\omega + \widetilde{w}^{(1)},
\end{equation*}
where $\widetilde{w}^{(1)}$ solves the following initial value problem on $[0,T_1] \times \bR^3$
\begin{equation} \label{equ:nlw_for_w_tilde_first_iteration}
 \left\{
 \begin{split} 
  -\widetilde{w}^{(1)}_{tt} + \Delta \widetilde{w}^{(1)} &= \bigl|v^{(1)} + u_{f, >N}^\omega + \widetilde{w}^{(1)} \bigr|^{\rho-1} \bigl(v^{(1)} + u_{f, >N}^\omega + \widetilde{w}^{(1)} \bigr) - |v^{(1)}|^{\rho-1} v^{(1)} , \\
  (\widetilde{w}^{(1)}, \widetilde{w}^{(1)}_t)|_{t=0} &= (0, 0). 
 \end{split} \right.
\end{equation}
Using \eqref{equ:strichartz_bounds_on_v} and the averaging effects \eqref{equ:D_K,N} for the free evolution $u_{f, > N}^\omega$, Lemma \ref{lemma:probabilistic_lwp} yields a unique solution 
\begin{equation*}
(\widetilde{w}^{(1)}, \widetilde{w}^{(1)}_t) \in C\bigl([0,T_1]; \dot{H}^1_x(\bR^3)\bigr) \cap L^{q(\rho)}_t L^{2 \rho}_x\bigl([0,T_1]\times\bR^3\bigr) \times C\bigl([0,T_1]; L^2_x(\bR^3)\bigr)
\end{equation*}
to \eqref{equ:nlw_for_w_tilde_first_iteration} provided $0 < c < 1$ in the definition of $T_1$ is chosen sufficiently small (independently of the size of $K$ and $N$) and $N$ is chosen sufficiently large. Moreover, we have 
\begin{equation} \label{equ:lwp_w_1_bounds}
 \begin{split}
 &\|\widetilde{w}^{(1)}(T_1)\|_{\dot{H}^1_x(\bR^3)} + \|\widetilde{w}^{(1)}_t(T_1)\|_{L^2_x(\bR^3)} + \|\widetilde{w}^{(1)}(T_1)\|_{L^{\rho + 1}_x(\bR^3)} \\ 
 &\quad \quad \lesssim T_1^{1 - \frac{\rho-1}{q(\rho)} -\varepsilon}  K^\rho N^{2\varepsilon + (1-s)\rho-1}.
 \end{split}
\end{equation}

\medskip

In the next step we build a solution $u^{(2)} = v^{(2)} + w^{(2)}$ to \eqref{equ:nlw_rho} on the time interval $[T_1, 2 T_1]$. As before, we would like to construct $v^{(2)}$ using the deterministic global existence theory at the energy level and construct $w^{(2)}$ through the probabilistic local well-posedness result from Lemma \ref{lemma:probabilistic_lwp}. To this end we note that $w^{(1)}$ is comprised of the free evolution $u_{f, >N}^\omega$, which is at low regularity, and the nonlinear component $\widetilde{w}^{(1)}$, which lies in the energy space by \eqref{equ:lwp_w_1_bounds}. As initial data for $v^{(2)}$ at time $T_1$ we therefore take the sum of $v^{(1)}(T_1)$ and $\widetilde{w}^{(1)}(T_1)$, and consider
\begin{equation} \label{equ:nlw_for_v_second_iteration}
\left\{ \begin{split}
  -v^{(2)}_{tt} + \Delta v^{(2)} &= | v^{(2)} |^{\rho-1} v^{(2)} \text{ on } [T_1, 2 T_1] \times \bR^3, \\
  (v^{(2)}, v^{(2)}_t)|_{t=T_1} &= (v^{(1)}(T_1) + \widetilde{w}^{(1)}(T_1), v^{(1)}_t(T_1) + \widetilde{w}^{(1)}_t(T_1)).
 \end{split}  \right.
\end{equation}
Once again, by the deterministic global theory, this initial value problem has a unique solution 
\begin{align*}
(v^{(2)}, v^{(2)}_t) \in C\bigl([T_1, 2 T_1]; \dot{H}^1_x(\bR^3) \times L^2_x(\bR^3)\bigr).
\end{align*}
Moreover, we obtain bounds on the $L^{q(\rho)}_t L^{2\rho}_x([T_1, 2 T_1]\times\bR^3)$-norm of~$v^{(2)}$ as in \eqref{equ:strichartz_bounds_on_v}. Using these bounds and the averaging effects \eqref{equ:D_K,N}, we apply Lemma~\ref{lemma:probabilistic_lwp} to solve the difference equation for $w^{(2)} = u^{(2)} - v^{(2)}$ on $[T_1, 2 T_1] \times \bR^3$,
\begin{equation} \label{equ:nlw_for_w_second_iteration}
 \left\{
 \begin{split}
  -w^{(2)}_{tt} + \Delta w^{(2)} &= \bigl| v^{(2)} + w^{(2)} \bigr|^{\rho-1} \bigl( v^{(2)} + w^{(2)} \bigr) - \bigl| v^{(2)} \bigr|^{\rho -1} v^{(2)} , \\
  (w^{(2)}, w^{(2)}_t)|_{t=T_1} &= \bigl( u_{f, >N}^\omega(T_1), \partial_t u_{f, > n}^\omega(T_1) \bigr).
 \end{split} \right.
\end{equation}
Note that Lemma \ref{lemma:probabilistic_lwp} can also be applied on the time interval $[T_1, 2 T_1]$ by time translation. We therefore find a solution $w^{(2)} = u_{f, >N}^\omega + \widetilde{w}^{(2)}$ to \eqref{equ:nlw_for_w_second_iteration} with 
\begin{align*}
(\widetilde{w}^{(2)}, \widetilde{w}^{(2)}_t) \in C\bigl([T_1, 2 T_1]; \dot{H}^1_x(\bR^3)\bigr) \cap L^{q(\rho)}_t L^{2 \rho}_x\bigl([T_1, 2 T_1] \times \bR^3\bigr)  \times C\bigl([T_1, 2T_1]; L^2_x(\bR^3)\bigr).
\end{align*}

\medskip

In order to obtain a solution $u$ to \eqref{equ:nlw_rho} on the whole time interval $[0, T]$, we iterate this procedure on consecutive intervals for $\lceil \frac{T}{T_1} \rceil$ times. At every step we redistribute the data as in \eqref{equ:nlw_for_v_second_iteration}. To make the process uniform and thus reach the time $T$, we have to take into account that this redistribution increases the energy at each step. To estimate this growth, we invoke energy conservation for the solution to \eqref{equ:nlw_for_v_first_iteration}. We have
\begin{equation*}
 E( v^{(1)}(T_1) + \widetilde{w}^{(1)}(T_1) ) = E(v^{(1)}(0)) + \Bigl( E(v^{(1)}(T_1) + \widetilde{w}^{(1)}(T_1)) - E(v^{(1)}(T_1)) \Bigr),
\end{equation*}
and
\begin{equation} \label{equ:energy_increment}
 \begin{split}
  &E(v^{(1)}(T_1) + \widetilde{w}^{(1)}(T_1)) - E(v^{(1)}(T_1)) \\
  &\lesssim \|v^{(1)}(T_1)\|_{\dot{H}^1_x(\bR^3)} \|\widetilde{w}^{(1)}(T_1)\|_{\dot{H}^1_x(\bR^3)} + \|v_t^{(1)}(T_1)\|_{L^2_x(\bR^3)} \|\widetilde{w}_t^{(1)}(T_1)\|_{L^2_x(\bR^3)} \\
  &\quad \quad + \|v^{(1)}(T_1)\|_{L^{\rho+1}_x(\bR^3)}^{\rho} \|\widetilde{w}^{(1)}(T_1)\|_{L^{\rho+1}_x(\bR^3)} + E(\widetilde{w}^{(1)}(T_1)) \\
  &\lesssim E(v^{(1)}(0))^{\frac{1}{2}} \bigl( \|\widetilde{w}^{(1)}(T_1)\|_{\dot{H}^1_x(\bR^3)} + \|\widetilde{w}_t^{(1)}(T_1)\|_{L^2_x(\bR^3)}\bigr) \\
  &\quad \quad + E(v^{(1)}(0))^{\frac{\rho}{\rho+1}} \|\widetilde{w}^{(1)}(T_1)\|_{L^{\rho+1}_x(\bR^3)} + E(\widetilde{w}^{(1)}(T_1)).
 \end{split}
\end{equation}
The term $E(v^{(1)}(0))^{\frac{\rho}{\rho+1}} \|\widetilde{w}^{(1)}(T_1)\|_{L^{\rho+1}_x(\bR^3)}$ gives the largest contribution to the energy increment \eqref{equ:energy_increment}. 
In light of \eqref{equ:energy_conservation} and \eqref{equ:lwp_w_1_bounds}, we must ensure that
\begin{equation*}
 \frac{T}{T_1} \Bigl( (K N^{1-s})^2 \Bigr)^{\frac{\rho}{\rho+1}} \Bigl( T_1^{1 - \frac{\rho-1}{q(\rho)} -\varepsilon} K^\rho N^{2\varepsilon +(1-s)\rho -1} \Bigr) \lesssim (K N^{1-s})^2.
\end{equation*}
Inserting $T_1 = c (K N^{1-s})^{-\frac{\rho-1}{\alpha(\rho)}}$, this is equivalent to
\begin{equation}
\label{equ:condition}
 T c^{-\frac{\rho-1}{q(\rho)}-\varepsilon} K^{ \frac{9\rho^2-6\rho-3}{(5-\rho)(\rho+1)\rho} +\varepsilon \frac{2(\rho-1)}{5-\rho} } N^{ \frac{\rho^3+5\rho^2-11\rho-3 - s (9\rho^2-6\rho-3)}{(5-\rho)(\rho+1)\rho} + \varepsilon ( (1-s)\frac{2(\rho-1)}{5-\rho} +2) } \lesssim 1.
\end{equation}
For any $s > \frac{\rho^3+5\rho^2-11\rho-3}{9\rho^2-6\rho-3}$, we can make the exponent on $N$ in \eqref{equ:condition} negative by fixing $\varepsilon > 0$ sufficiently small at the beginning (depending on the values of $s$ and $\rho$, which are fixed during the course of the argument). Hence, taking $N \equiv N(K,T)$ sufficiently large we can ensure that condition \eqref{equ:condition} is satisfied. This completes the proof of Proposition~\ref{prop:rho}.
\end{proof}

We now present the proof of Lemma~\ref{lemma:probabilistic_lwp}.

\begin{proof}[Proof of Lemma~\ref{lemma:probabilistic_lwp}]
In this proof we are working on $[0,T_1] \times \bR^3$ and will omit this notation. For $\widetilde{w} \in L^{q(\rho)}_t L^{2 \rho}_x$ define the map
 \begin{equation*}
  \Phi(\widetilde{w})(t) = - \int_0^t \frac{\sin((t-s)|\nabla|)}{|\nabla|} \left( \bigl| v + u_{f, >N}^\omega + \widetilde{w} \bigr|^{\rho -1} \bigl( v + u_{f, >N}^\omega + \widetilde{w} \bigr) - |v|^{\rho -1} v \right)(s) \, ds.
 \end{equation*}
 Using the general inequality
 \begin{equation*}
  \bigl| |z_1|^{\rho-1} z_1 - |z_2|^{\rho -1} z_2 \bigr| \lesssim |z_1 - z_2| ( |z_1|^{\rho -1} + |z_2|^{\rho -1} ) \text{ for any } z_1, z_2 \in \bC
 \end{equation*}
 and the Strichartz estimates \eqref{equ:strichartz_estimates} at regularity $\gamma =1$ we bound
 \begin{equation} \label{equ:Phi_estimate_into_ball}
  \begin{split}
   & \| \Phi(\widetilde{w}) \|_{L^{\infty}_t \dot{H}^1_x} + \|\partial_t \Phi(\widetilde{w})\|_{L^{\infty}_t L^2_x} + \| \Phi(\widetilde{w}) \|_{L^{q(\rho)}_t L^{2\rho}_x} \\
   &\lesssim \Bigl\| \bigl| v + u_{f, > N}^\omega + \widetilde{w} \bigl|^{\rho -1} (v + u_{f, >N}^\omega + \widetilde{w}) - |v|^{\rho-1} v \bigr| \Bigr\|_{L^1_t L^2_x} \\
   &\lesssim \Bigl\| | u_{f, >N}^\omega +\widetilde{w} | ( |v + u_{f, >N}^\omega + \widetilde{w}|^{\rho-1} + |v|^{\rho -1} ) \Bigr\|_{L^1_t L^2_x} \\
   &\lesssim \Bigl\| (|u_{f, >N}^\omega| + |\widetilde{w}|) (|v|^{\rho-1} + |u_{f, >N}^\omega|^{\rho -1} + |\widetilde{w}|^{\rho-1}) \Bigr\|_{L^1_t L^2_x} \\
   &\lesssim \bigl\| u_{f, >N}^\omega \bigr\|_{L^\rho_t L^{2\rho}_x}^\rho + \|\widetilde{w}\|_{L^\rho_t L^{2\rho}_x}^\rho + \bigl\| |u_{f, >N}^\omega| |v|^{\rho-1} \bigr\|_{L^1_t L^2_x} \\
   &\quad + \bigl\| |\widetilde{w}| |v|^{\rho-1} \bigr\|_{L^1_t L^2_x} + \bigl\| |u_{f, >N}^\omega|^{\rho-1} |\widetilde{w}| \bigr\|_{L^1_t L^2_x} \\
   &= I + II + III + IV + V.
  \end{split}
 \end{equation}
 We now estimate the terms $I - V$ separately.

\medskip

\noindent {\bf Term I:} By H\"older's inequality in time and \eqref{equ:lwp_u_bound}, we obtain
\begin{equation*} 
 \begin{split}
  \bigl\| u_{f, >N}^\omega \bigr\|^\rho_{L^\rho_t L^{2\rho}_x}  &\leq T_1^{1-\rho \varepsilon} \bigl\|u_{f,>N}^\omega\bigr\|^\rho_{L^{1/\varepsilon}_t L^{2\rho}_x} \leq T_1^{1-\rho \varepsilon} (K N^{-s + 2\varepsilon})^\rho.
 \end{split}
\end{equation*}

\noindent {\bf Term II:} By H\"older's inequality in time, we have
\begin{equation*}
 \|\widetilde{w}\|^\rho_{L^\rho_t L^{2\rho}_x} \leq T_1^{\alpha(\rho)} \|\widetilde{w}\|^\rho_{L^{q(\rho)}_t L^{2\rho}_x}.
\end{equation*}

\noindent {\bf Term III:}  Using \eqref{equ:lwp_v_bound} and \eqref{equ:lwp_u_bound}, we find
\begin{equation*} 
 \begin{split}
  \bigl\| |u_{f, >N}^\omega| |v|^{\rho-1} \bigr\|_{L^1_t L^2_x} &\leq \bigl\| u_{f, >N}^\omega \bigr\|_{L^\rho_t L^{2\rho}_x} \|v\|^{\rho-1}_{L^\rho_t L^{2\rho}_x} \\
  &\leq T_1^{1/\rho-\varepsilon} \bigl\|u_{f,>N}^\omega\bigr\|_{L^{1/\varepsilon}_t L^{2\rho}_x} T_1^{\alpha(\rho) (\rho - 1)/\rho} \|v\|^{\rho-1}_{L^{q(\rho)}_t L^{2\rho}_x} \\
  &\lesssim T_1^{\frac{1 + \alpha(\rho)(\rho - 1)}{\rho} -\varepsilon}  (K N^{-s + 2 \varepsilon}) (K N^{1-s})^{\rho-1} \\
  &\lesssim T_1^{1 - \frac{\rho-1}{q(\rho)} -\varepsilon}  K^\rho N^{2\varepsilon + (1-s)\rho-1}.
 \end{split}
\end{equation*}

\noindent {\bf Term IV:} By \eqref{equ:lwp_v_bound} we have
\begin{equation*} 
 \begin{split}
  \bigl\| |\widetilde{w}| |v|^{\rho-1} \bigr\|_{L^1_t L^2_x} &\leq T_1^{\alpha(\rho)/\rho} \bigl\| \widetilde{w} \bigr\|_{L^{q(\rho)}_t L^{2\rho}_x}T_1^{\alpha(\rho)(\rho-1)/\rho}   \|v\|^{\rho-1}_{L^{q(\rho)}_t L^{2\rho}_x} \\
  &\lesssim T_1^{\alpha(\rho)}  \bigl\| \widetilde{w} \bigr\|_{L^{q(\rho)}_t L^{2\rho}_x} (K N^{1-s})^{\rho-1}.
 \end{split}
\end{equation*}

\noindent {\bf Term V:} Using \eqref{equ:lwp_u_bound} we bound
\begin{equation*} 
 \begin{split}
  \bigl\| |u_{f, >N}^\omega|^{\rho-1} |\widetilde{w}| \bigr\|_{L^1_t L^2_x} &\leq \bigl\| u_{f, >N}^\omega \bigr\|^{\rho-1}_{L^\rho_t L^{2\rho}_x} \| \widetilde{w} \|_{L^\rho_t L^{2\rho}_x} \\
  &\leq T_1^{(\rho-1)/\rho - (\rho-1)\varepsilon} \bigl\| u_{f, >N}^\omega \bigr\|^{\rho-1}_{L^{1/\varepsilon}_t L^{2\rho}_x} T_1^{\alpha(\rho)/\rho} \|\widetilde{w}\|_{L^{q(\rho)}_t L^{2\rho}_x} \\
  &= T_1^{\frac{\rho-1 + \alpha(\rho)}{\rho} -(\rho-1)\varepsilon}\bigl\| u_{f, >N}^\omega \bigr\|^{\rho-1}_{L^{1/\varepsilon}_t L^{2\rho}_x}\|\widetilde{w}\|_{L^{q(\rho)}_t L^{2\rho}_x} \\   
  &\leq T_1^{1 - \frac{1}{q(\rho)} -(\rho-1)\varepsilon}  (K N^{-s + 2 \varepsilon})^{\rho-1} \|\widetilde{w}\|_{L^{q(\rho)}_t L^{2\rho}_x}.
 \end{split}
\end{equation*}

\noindent Collecting terms we obtain
\begin{equation} \label{equ:collecting_terms}
 \begin{split}
  &\| \Phi(\widetilde{w}) \|_{L^{\infty}_t \dot{H}^1_x} + \| \partial_t \Phi(\widetilde{w})\|_{L^{\infty}_t L^2_x} + \|\Phi(\widetilde{w})\|_{L^{q(\rho)}_t L^{2\rho}_x} \\
  &\quad \lesssim T_1^{1-\rho\varepsilon} K^\rho N^{-\rho s + 2\rho \varepsilon} +  T_1^{\alpha(\rho)} \|\widetilde{w}\|^\rho_{L^{q(\rho)}_t L^{2\rho}_x} \\
  & \qquad + T_1^{1 - \frac{\rho-1}{q(\rho)} -\varepsilon}  K^\rho N^{2\varepsilon + (1-s)\rho-1}  + T_1^{\alpha(\rho)}  (K N^{1-s})^{\rho-1} \bigl\| \widetilde{w} \bigr\|_{L^{q(\rho)}_t L^{2\rho}_x} \\
  & \qquad + T_1^{1 - \frac{1}{q(\rho)} -(\rho-1)\varepsilon}  (K N^{-s + 2 \varepsilon})^{\rho-1} \|\widetilde{w}\|_{L^{q(\rho)}_t L^{2\rho}_x}.
   \end{split}
\end{equation}
It follows that by choosing $0 < c < 1$ sufficiently small (independently of $K$ and $N$) and $N$ sufficiently large, $\Phi$ maps a ball $B$ of radius $R(K,N) > 0$ with respect to the $L^{q(\rho)}_t L^{2\rho}_x$-norm into itself, where
\begin{equation} \label{equ:ball_radius_estimate}
 R(K, N) \lesssim T_1^{1 - \frac{\rho-1}{q(\rho)} -\varepsilon}  K^\rho N^{2\varepsilon + (1-s)\rho-1}.
\end{equation}
In a similar vein, we show that $\Phi$ is a contraction on $B$ with respect to the $L^{q(\rho)}_t L^{2\rho}_x$-norm for $0 < c < 1$ sufficiently small and $N$ sufficiently large. Thus, $\Phi$ has a unique fixed point $\widetilde{w} \in B$ and $\widetilde{w}$ is the unique solution to \eqref{equ:lwp_nlw}.

In order to obtain \eqref{equ:lwp_w_bounds} it remains to estimate the $L^{\rho+1}_x(\bR^3)$-norm of $\widetilde{w}$ at time $T_1$. By Sobolev embedding and Minkowski's integral inequality we find 
\begin{equation*}
 \begin{split}
  &\|\widetilde{w}(T_1)\|_{L^{\rho+1}_x(\bR^3)} \\
  &\lesssim \|\widetilde{w}(T_1)\|_{H^1_x(\bR^3)} \\
  &\lesssim \Bigl\| \int_0^{T_1} \frac{\sin((T_1-s)|\nabla|)}{|\nabla|} \bigl( |v + u_{f, >N}^\omega + \widetilde{w}|^{\rho -1} (v + u_{f, >N}^\omega + \widetilde{w}) - |v|^{\rho -1} v \bigr)(s) \, ds \Bigr\|_{H^1_x(\bR^3)} \\
  &\lesssim \int_0^{T_1} \Bigl\| \frac{\sin((T_1-s)|\nabla|)}{|\nabla|} \bigl( |v + u_{f, >N}^\omega + \widetilde{w}|^{\rho -1} (v + u_{f, >N}^\omega + \widetilde{w}) - |v|^{\rho -1} v \bigr)(s) \Bigr\|_{H^1_x(\bR^3)} \, ds \\
  &\lesssim \int_0^{T_1} (1 + T_1 -s) \Bigl\| \bigl( |v + u_{f, >N}^\omega + \widetilde{w}|^{\rho -1} (v + u_{f, >N}^\omega + \widetilde{w}) - |v|^{\rho -1} v \bigr)(s) \Bigr\|_{L^2_x(\bR^3)} \, ds \\
  &\lesssim (1+T_1) \Bigl\| |v + u_{f, >N}^\omega + \widetilde{w}|^{\rho -1} (v + u_{f, >N}^\omega + \widetilde{w}) - |v|^{\rho -1} v \Bigr\|_{L^1_t L^2_x} \\
  &\lesssim \Bigl\| |v + u_{f, >N}^\omega + \widetilde{w}|^{\rho -1} (v + u_{f, >N}^\omega + \widetilde{w}) - |v|^{\rho -1} v \Bigr\|_{L^1_t L^2_x},
 \end{split}
\end{equation*}
where we used that $T_1 \leq 1$. From \eqref{equ:Phi_estimate_into_ball}, \eqref{equ:collecting_terms} and \eqref{equ:ball_radius_estimate} we infer
\begin{equation*} 
 \|\widetilde{w}(T_1)\|_{L^{\rho+1}_x(\bR^3)} \lesssim R(K, N) \lesssim T_1^{1 - \frac{\rho-1}{q(\rho)} -\varepsilon}  K^\rho N^{2\varepsilon + (1-s)\rho-1}.
\end{equation*}
This completes the proof of Lemma \ref{lemma:probabilistic_lwp}.
\end{proof}

Finally, we address the uniqueness statements in Theorem \ref{thm:rho} and Proposition \ref{prop:rho}. 

\begin{remark} \label{remark:uniqueness}
Analogously to \cite[Remark 1.2]{CO}, uniqueness for the ``low frequency part" $(v^{(j)}, v^{(j)}_t)$ in the $j$-th step holds in the space $C\bigl([(j-1)T_1, j T_1]; H^1_x(\bR^3) \times L^2_x(\bR^3)\bigr)$. However, uniqueness for the ``high frequency part" $(w^{(j)}, w^{(j)}_t)$ in the $j$-th step only holds in the ball centered at $u_{f, >N}^\omega(t)$ of small radius in $L^{q(\rho)}_t L^{2 \rho}_x\bigl([(j-1) T_1, j T_1] \times \bR^3\bigr)$.
\end{remark}

\section{Proof of Theorem \ref{thm:quintic}} \label{sec:quintic}

In this section we prove the almost sure global existence result for the quintic nonlinear wave equation from Theorem \ref{thm:quintic}. 

\begin{proof}[Proof of Theorem \ref{thm:quintic}]
 We first derive probabilistic a priori estimates on the $L^5_t L^{10}_x(\bR \times \bR^3)$ norm of the free evolution $u_f^\omega$ and then use these to construct global solutions to \eqref{equ:quintic_nlw} through a suitable fixed point argument.

 Since we have by assumption that $\frac{2}{5} \leq s < 1$, there exists $\frac{10}{3} \leq r < 10$ such that the exponent pair $(5, r)$ is wave-admissible and Strichartz-admissible at regularity $s$. Similarly to the proof of Lemma \ref{lemma:averaging_effects_rho}, using the large deviation estimate \eqref{equ:large_deviation_estimate}, the unit-scale Bernstein estimate \eqref{equ:bernstein_estimate} and the Strichartz estimate \eqref{equ:strichartz_estimates} at regularity $s$, we obtain for any $p \geq 10$ 
 \begin{equation*}
  \begin{split}
   &\|u_f^\omega\|_{L^p_\omega(\Omega; L^5_t  L^{10}_x(\bR \times \bR^3))} \\
   &\lesssim \sqrt{p} \Bigl( \sum_{k \in \bZ^3} \bigl\|\cos(t|\nabla|) P_k f_1\bigr\|^2_{L^5_t L^{10}_x(\bR \times \bR^3)} \Bigr)^{1/2} + \sqrt{p} \Bigl( \sum_{k \in \bZ^3} \Bigl\| \frac{\sin(t|\nabla|)}{|\nabla|} P_k f_2 \Bigr\|^2_{L^5_t L^{10}_x(\bR \times \bR^3)} \Bigr)^{1/2} \\
   &\lesssim \sqrt{p} \Bigl( \sum_{k \in \bZ^3} \bigl\|\cos(t|\nabla|) P_k f_1\bigr\|^2_{L^5_t L^r_x(\bR \times \bR^3)} \Bigr)^{1/2} + \sqrt{p} \Bigl( \sum_{k \in \bZ^3} \Bigl\| \frac{\sin(t|\nabla|)}{|\nabla|} P_k f_2 \Bigr\|^2_{L^5_t L^r_x(\bR \times \bR^3)} \Bigr)^{1/2} \\
   &\lesssim \sqrt{p} \Bigl( \sum_{k \in \bZ^3} \|P_k f_1\|^2_{\dot{H}^s_x(\bR^3)} \Bigr)^{1/2} + \sqrt{p} \Bigl( \sum_{k \in \bZ^3} \|P_k f_2\|^2_{\dot{H}^{s-1}_x(\bR^3)} \Bigr)^{1/2} \\
   &\lesssim \sqrt{p} \bigl( \|f_1\|_{H^s_x(\bR^3)} + \|f_2\|_{\dot{H}^{s-1}_x(\bR^3)} \bigr).
  \end{split}
 \end{equation*}
 Then Lemma \ref{lem:tzvetkov_probest} implies that there exist absolute constants $C, c > 0$ such that
 \begin{equation} \label{equ:probability_estimate_event_quintic}
  \bP \bigl( \omega \in \Omega : \|u_f^\omega\|_{L^5_t L^{10}_x(\bR \times \bR^3)} > K \bigl) \leq C e^{-c K^2/\|f\|^2_{H^s_x \times \dot{H}^{s-1}_x}}
 \end{equation}
 for every $K > 0$.

 We now look for global solutions of the form $u = u_f^\omega + w$ to \eqref{equ:quintic_nlw}, where $w$ has to satisfy the nonlinear wave equation
 \begin{equation} \label{equ:quintic_proof_difference_equation} 
  \left\{ \begin{split}
   -w_{tt} + \Delta w &= |u_f^\omega + w|^4 (u_f^\omega + w) \text{ on } \bR \times \bR^3, \\
   (w, w_t)|_{t = 0} &= (0, 0).
  \end{split} \right.
 \end{equation}
It is straightforward to see that there exists $\varepsilon > 0$ such that if $\|u_f^\omega\|_{L^5_t L^{10}_x(\bR \times \bR^3)} \leq \varepsilon$, then the map 
\begin{equation*}
 \Phi(w)(t) := - \int_0^t \frac{\sin((t-s)|\nabla|)}{|\nabla|} \bigl( |u_f^\omega + w|^4 (u_f^\omega + w) \bigr)(s) \, ds 
\end{equation*}
is a contraction on a ball of radius $\varepsilon$ with respect to the $L^5_t L^{10}_x(\bR \times \bR^3)$ norm. Its unique fixed point 
\begin{equation*}
 (w, w_t) \in C(\bR; \dot{H}^1_x(\bR^3)) \cap L^5_t L^{10}_x(\bR \times \bR^3) \times C(\bR; L^2_x(\bR^3))
\end{equation*}
is the global solution to \eqref{equ:quintic_proof_difference_equation}. The probability estimate \eqref{equ:probabilistic_estimate_quintic_thm} on the event $\Omega_f$ in the statement of Theorem \ref{thm:quintic} follows immediately from \eqref{equ:probability_estimate_event_quintic}, while the scattering statement follows readily using that $\bigl\| |u_f^\omega + w|^5 \bigr\|_{L^1_t L^2_x(\bR \times \bR^3)} < \infty$ for every $\omega \in \Omega_f$.
\end{proof}

\bibliographystyle{myamsplain}
\bibliography{refs}

\end{document}